\documentclass[12pt]{article}
\usepackage{amssymb,amsmath,amsthm}
\usepackage[shortlabels]{enumitem}

\usepackage{xcolor}

\newcommand{\Z}{{\mathbb Z}}

\newcommand{\OO}{\mathcal{O}}

\newcommand{\Q}{\mathbb{Q}}
\newcommand{\R}{\mathbb{R}}
\newcommand{\F}{\mathbb{F}}
\newcommand{\M}{\mathcal{M}}

\newcommand{\xv}{\vec{x}}
\newcommand{\yv}{\vec{y}}
\newcommand{\zv}{\vec{z}}
\newcommand{\av}{\vec{a}}
\newcommand{\bv}{\vec{b}}
\newcommand{\cv}{\vec{c}}
\newcommand{\dv}{\vec{d}}
\newcommand{\zerov}{\vec{0}}

\newcommand{\Kbar}{\overline{K}}
\newcommand{\Lbar}{\overline{L}}

\newcommand{\ra}{\rightarrow}

\DeclareMathOperator{\Gal}{Gal}
\DeclareMathOperator{\ch}{char}

\newtheorem{theorem}{Theorem}
\newtheorem{lemma}[theorem]{Lemma}
\newtheorem{prop}[theorem]{Proposition}
\newtheorem{cor}[theorem]{Corollary}

\theoremstyle{definition}
\newtheorem{remark}[theorem]{Remark}

\numberwithin{equation}{section}
\numberwithin{theorem}{section}

\title{Artin-Schreier-Witt extensions and ramification
breaks}
\author{G. Griffith Elder \\
Department of Mathematics \\
University of Connecticut \\
Storrs, CT 06269 \\
USA \\[.2cm]
{\tt griff.elder@uconn.edu}
\and
Kevin Keating \\
Department of Mathematics \\
University of Florida \\
Gainesville, FL 32611 \\
USA \\[.2cm]
{\tt keating@ufl.edu}}

\begin{document}

\maketitle

\begin{abstract}
Let $K=k((t))$ be a local field of characteristic
$p>0$, with perfect residue field $k$.  Let
$\av=(a_0,a_1,\dots,a_{n-1})\in W_n(K)$ be a Witt vector
of length $n$.  Artin-Schreier-Witt theory associates
to $\av$ a cyclic extension $L/K$ of degree $p^i$ for
some $i\le n$.  Assume that the vector $\av$ is
``reduced'',
and that $v_K(a_0)<0$; then $L/K$ is a totally ramified
extension of degree $p^n$.  In the case where $k$ is
finite, Kanesaka-Sekiguchi and Thomas used class field
theory to explicitly compute the upper ramification
breaks of $L/K$ in terms of the valuations of the
components of $\av$.  In this note we use a direct
method to show that these formulas remain valid when $k$
is an arbitrary perfect field.
\end{abstract}

Let $K$ be a local field of characteristic $p$ with
perfect residue field $k$.  Then $K\cong k((t))$.  Let
$n\ge1$ and let $W_n(K)$ be the ring of Witt vectors of
length $n$ over $K$.  Choose $\av\in W_n(K)$ which is
reduced in the sense of \cite{lt}.  Write
$\av=(a_0,a_1,\dots,a_{n-1})$ and assume that
$v_K(a_0)<0$.  Let $F:K\ra K$ be the $p$th power map;
then $F$ induces a homomorphism from $W_n(K)$ to itself
by acting on components.  Let $K^{sep}$ be a separable
closure of $K$ and suppose that
$\xv=(x_0,x_1,\dots,x_{n-1})\in W_n(K^{sep})$ satisfies
$F(\xv)=\xv\oplus\av$, where $\oplus$ denotes Witt
vector addition.  Set $K_n=K(x_0,\ldots,x_{n-1})$.
Then $K_n/K$ is a totally ramified $\Z/p^n\Z$-extension.
In this note we compute the upper ramification breaks of
$K_n/K$ in terms of the valuations of the components of
$\av$.  This generalizes work of Kanesaka-Sekiguchi
\cite{ks}, and Thomas \cite{lt}, who proved that the
formula is valid in the case where $k$ is finite.

     In the first section we recall the definition and
basic properties of Witt vectors.  We then prove a basic
result about the Galois characters associated to Witt
vectors over a field $K$ of characteristic $p$.  This
result will be used in the second section to compute the
upper ramification breaks of $(\Z/p^n\Z)$-extensions
$K_n/K$ in the case where $K$ is a local field of
characteristic $p$.

\section{Witt vectors and cyclic $p$-extensions}

Let $K$ be a field of characteristic $p>0$.  For $i\ge0$
define the $i$th phantom coordinate polynomial by
\begin{align*}
\phi_i(X_0,X_1,\dots,X_i)&=X_0^{p^i}+pX_1^{p^{i-1}}
+\dots+p^{i-1}X_{i-1}^p+p^iX_i \\
&\in\Z[X_0,X_1,\dots,X_i].
\end{align*}
Let $R$ be a commutative ring with 1.  The idea behind
the Witt vectors is to define polynomial operations on
$R^n$ so that the phantom coordinates give ring
homomorphisms from $R^n$ to $R$.  The Witt addition and
multiplication polynomials $S_i,M_i$ are elements of
$\Q[X_0,\dots,X_i,Y_0\dots,Y_i]$ defined recursively by
the relations
\begin{align*}
\phi_i(S_0,S_1,\dots,S_i)&=\phi_i(X_0,X_1,\dots,X_i)
+\phi_i(Y_0,Y_1,\dots,Y_i) \\
\phi_i(M_0,M_1,\dots,M_i)&=\phi_i(X_0,X_1,\dots,X_i)\cdot
\phi_i(Y_0,Y_1,\dots,Y_i).
\end{align*}
It is clear from this definition that $S_i$ and $M_i$
have coefficients in $\Z[1/p]$.  In Satz~1 of \cite{W}
Witt showed that $S_i$ and $M_i$ actually have
coefficients in $\Z$.  Therefore for $\av,\bv\in R^n$ we
can define
\begin{align*}
\av\oplus\bv&=(a_0,a_1,\dots,a_{n-1})\oplus
(b_0,b_1,\dots,b_{n-1}) \\
&=(S_0(a_0,b_0),S_1(a_0,a_1,b_0,b_1),\dots,
S_{n-1}(a_0,\dots,a_{n-1},b_0,\dots,b_{n-1})) \\
\av\odot\bv&=(a_0,a_1,\dots,a_{n-1})\odot
(b_0,b_1,\dots,b_{n-1}) \\
&=(M_0(a_0,b_0),M_1(a_0,a_1,b_0,b_1),\dots,
M_{n-1}(a_0,\dots,a_{n-1},b_0,\dots,b_{n-1}))
\end{align*}
These operations make $R^n$ a commutative ring with 1.
We say that $R^n$ with these operations is the ring of
Witt vectors of length $n$ over $R$, denoted by
$W_n(R)$.  Subtraction in $W_n(R)$ is represented by
$\ominus$.  For more information, see Chapter~III of
\cite{pdg}.

     Witt vectors can be used to give a generalization
of Artin-Schreier theory for cyclic extensions of degree
$p^n$, as explained in Satz~12 of \cite{W}.  (See also
Th\'{e}or\`{e}me~2.3 of \cite{ltthesis}.) Let $K$ be
a field of characteristic $p$ and let $\av\in W_n(K)$.
Let $\xv\in W_n(K^{sep})$ satisfy $F(\xv)=\xv\oplus\av$.
Set $G_K=\Gal(K^{sep}/K)$ and let $\sigma\in G_K$.  Then
$\sigma(\xv)\ominus\xv\in W_n(\F_p)\cong\Z/p^n\Z$, so we
may define $\chi_{\av}^K:G_K\ra\Z/p^n\Z$ by setting
$\chi_{\av}^K(\sigma)=\sigma(\xv)\ominus\xv$.  Then
$\chi_{\av}^K$ is a continuous homomorphism from $G_K$ to
$\Z/p^n\Z$ which does not depend on the choice of $\xv$.
Conversely, for every continuous homomorphism
$\theta:G_K\ra\Z/p^n\Z$ there exists $\av\in W_n(K)$
such that $\theta=\chi_{\av}^K$.  For
$\av,\bv\in W_n(K)$ we have
$\chi_{\av\oplus\bv}^K=\chi_{\av}^K+\chi_{\bv}^K$.  In
addition, $\chi_{\av}^K=\chi_{\bv}^K$ if and only if
$\bv=\av\oplus F(\cv)\ominus\cv$ for some
$\cv\in W_n(K)$.  Let $K[\av]$ denote the fixed field of
$\ker(\chi_{\av}^K)\le G_K$.  It follows from the above
that $K[\av\oplus\bv]$ is contained in the compositum of
$K[\av]$ and $K[\bv]$.

     For $\av\in W_n(K)$ and $1\le i\le n-1$ write
\begin{align*}
\av&=(a_0,\dots,a_{i-1},a_i,\dots,a_{n-1}) \\
\mu_i(\av)&=(a_0,\dots,a_{i-1},0,\dots,0) \\
\lambda_i(\av)&=(0,\dots,0,a_i,\dots,a_{n-1}).
\end{align*}
Then $\av=\lambda_i(\av)+\mu_i(\av)$, where $+$
represents ordinary vector addition.  It follows from
the definition of Witt vector addition that we also have
$\av=\lambda_i(\av)\oplus\mu_i(\av)$.  Hence
$\chi_{\av}^K=\chi_{\lambda_i(\av)}^K+\chi_{\mu_i(\av)}^K$,
so $K[\av]$ is contained in the compositum of
$K[\lambda_i(\av)]$ and $K[\mu_i(\av)]$.  The maps
$\mu_i$ and $\lambda_i$ are not homomorphisms, but we do
have $\mu_i(\av\oplus\bv)
=\mu_i(\mu_i(\av)\oplus\mu_i(\bv))$ for
$\av,\bv\in W_n(K)$.

     Let $1_n=(1,0,\dots,0)$ be the multiplicative
identity element of $W_n(K)$.  Since $\ch(K)=p$ it
follows from Satz~2 of \cite{W} that for $\av\in W_n(K)$
we have
\[p\cdot\av=p\cdot(a_0,a_1,\dots,a_{n-1})
=(0,a_0^p,\dots,a_{n-2}^p).\]
Hence for $1\le i\le n-1$ we get
\begin{equation} \label{pione}
p^i\cdot1_n=(\underbrace{0,\dots,0}_i,1,
\underbrace{0,\dots,0}_{n-i-1}).
\end{equation}
Thus $p^i\cdot1_n=\lambda_i(p^i\cdot1_n)$.

     For $1\le j\le n-1$ let
$\eta_{j,n}:\Z/p^j\Z\ra\Z/p^n\Z$ be the homomorphism
which maps $r+p^j\Z$ to $p^{n-j}r+p^n\Z$.  The following
fact is probably well-known:

\begin{lemma} \label{shift}
Let $1\le i\le n-1$ and let
$\cv=(c_0,\dots,c_{n-i-1})\in W_{n-i}(K)$.  Set
\[\bv=(\underbrace{0,\dots,0}_i,c_0,\dots,c_{n-i-1})
\in W_n(K).\]
Then $\chi_{\bv}^K=\eta_{n-i,n}\circ\chi_{\cv}^K$.
\end{lemma}

\begin{proof}
It suffices to prove the lemma in the case $i=1$.  Let
$\zv\in W_{n-1}(K^{sep})$ satisfy $F(\zv)=\zv\oplus\cv$.
Write $\zv=(z_0,\dots,z_{n-2})$ and set
$\yv=(0,z_0,\dots,z_{n-2})\in W_n(K^{sep})$.  By the
definition of the Witt addition polynomials we have
\begin{equation} \label{Sj}
S_j(X_0,\dots,X_j,Y_0,\dots,Y_j)
=S_{j+1}(0,X_0,\dots,X_j,0,Y_0,\dots,Y_j).
\end{equation}
It follows that $F(\yv)=\yv\oplus\bv$.  Let
$\sigma\in G_K$.  Then $\sigma(\zv)=\zv$ if and only if
$\sigma(\yv)=\yv$.  Suppose
$\chi_{\cv}^K(\sigma)=p^j+p^{n-1}\Z$ for some
$0\le j\le n-2$.  Then
$\sigma(\zv)=\zv\oplus(p^j\cdot1_{n-1})$.  Hence by
(\ref{Sj}) and (\ref{pione}) we get
$\sigma(\yv)=\yv\oplus(p^{j+1}\cdot1_n)$.  Therefore
$\chi_{\bv}^K(\sigma)=p^{j+1}\cdot1_n
=\eta_{n-1,n}(\chi_{\cv}^K(\sigma))$.  It follows that
$\chi_{\bv}^K(\sigma)=\eta_{n-1,n}(\chi_{\cv}^K(\sigma))$
holds for general $\sigma\in G_K$.
\end{proof}

     Let $\av\in W_n(K)$ and let $\xv\in W_n(K^{sep})$
satisfy $F(\xv)=\xv\oplus\av$.  Let $1\le i\le n$ and
set $K_i=K(x_0,\dots,x_{i-1})$.  Then
$\chi_{\av}^K\big|_{G_{K_i}}=\chi_{\av}^{K_i}$.  The
following proposition gives $\bv\in W_n(K_i)$ such that
$\chi_{\av}^K\big|_{G_{K_i}}=\chi_{\bv}^{K_i}$ and the
first $i$ components of $\bv$ are all 0.

\begin{prop} \label{subext}
Let $\av\in W_n(K)$ and let $\xv\in W_n(K^{sep})$
satisfy $F(\xv)=\xv\oplus\av$.  Set
$K_i=K(x_0,\dots,x_{i-1})$.  Then
$\chi_{\av}^K\big|_{G_{K_i}}=\chi_{\bv}^{K_i}$, where
$\bv=\lambda_i(\mu_i(\xv)\oplus\av)\in W_n(K_i)$.
\end{prop}

\begin{proof}
Let $\cv=\av\oplus\mu_i(\xv)\ominus F(\mu_i(\xv))$.
Since $\mu_i(\xv)\in W_n(K_i)$ we get
$\chi_{\av}^K\big|_{G_{K_i}}=\chi_{\cv}^{K_i}$.  In
addition, we have
\begin{align*}
\cv&=\lambda_i(\av)\oplus\mu_i(\av)\oplus\mu_i(\xv)
\ominus\mu_i(F(\xv)) \\
&=\lambda_i(\av)\oplus\mu_i(\av)\oplus\mu_i(\xv)
\ominus\mu_i(\av\oplus\xv) \\
&=\lambda_i(\av)\oplus\mu_i(\av)\oplus\mu_i(\xv)
\ominus\mu_i(\mu_i(\av)\oplus\mu_i(\xv)) \\
&=\lambda_i(\av)\oplus
\lambda_i(\mu_i(\av)\oplus\mu_i(\xv)) \\
&=\lambda_i(\lambda_i(\av)\oplus
\lambda_i(\mu_i(\av)\oplus\mu_i(\xv))).
\end{align*}
Since $\lambda_i(\cv)=\lambda_i(\cv+\mu_i(\dv))$ for
$\cv,\dv\in W_n(K_i)$, it follows that
\begin{align*}
\cv&=\lambda_i((\lambda_i(\av)\oplus
\lambda_i(\mu_i(\av)\oplus\mu_i(\xv)))
+\mu_i(\mu_i(\av)\oplus\mu_i(\xv))) \\
&=\lambda_i(\lambda_i(\av)\oplus
\lambda_i(\mu_i(\av)\oplus\mu_i(\xv))
\oplus\mu_i(\mu_i(\av)\oplus\mu_i(\xv))) \\
&=\lambda_i(\lambda_i(\av)\oplus\mu_i(\av)\oplus
\mu_i(\xv)) \\
&=\lambda_i(\av\oplus\mu_i(\xv))=\bv.
\end{align*}
Hence $\chi_{\bv}^{K_i}=\chi_{\cv}^{K_i}
=\chi_{\av}^{K_i}=\chi_{\av}^K\big|_{G_{K_i}}$.
\end{proof}

     Combining Lemma~\ref{shift} with the proposition
gives the following:

\begin{cor} \label{shiftcor}
Let $\bv=\lambda_i(\mu_i(\xv)\oplus\av)\in W_n(K_i)$, so
that
\[\bv=(\underbrace{0,\dots,0}_i,d_0,\dots,d_{n-i-1}).\]
Set $\dv=(d_0,\dots,d_{n-i-1})\in W_{n-i}(K_i)$.
Then $\chi_{\av}^K\big|_{G_{K_i}}
=\eta_{n-i,n}\circ\chi_{\dv}^{K_i}$.
\end{cor}

\section{Ramification in cyclic $p$-extensions of local
fields}

Let $K$ be a local field of characteristic $p$ with
perfect residue field.  We begin this section with a
review of higher ramification theory for finite Galois
extensions of $K$.  We then apply this theory to the
case of cyclic $p$-extensions.  For general facts about
higher ramification theory see Chapter~IV of \cite{cl}.

     Let $v_K$ be the valuation on $K$ normalized so
that $v_K(K^{\times})=\Z$.  Let
$\OO_K=\{x\in K:v_K(x)\ge0\}$ be the ring of integers of
$K$ and let $\M_K=\{x\in K:v_K(x)\ge1\}$ be the maximal
ideal of $\OO_K$.  Let $\pi_K\in\OO_K$ satisfy
$v_K(\pi_K)=1$; then $\M_K=\pi_K\OO_K$.  Let
$\Kbar=\OO_K/\M_K$ be the residue field of $K$.  Let
$L/K$ be a finite extension; then $v_L$, $\OO_L$,
$\pi_L$, $\M_L$, and $\Lbar$ are also defined.  The
ramification index of $L/K$ is $e_{L/K}=v_L(\pi_K)$ and
the residue class degree is $f_{L/K}=[\Lbar:\Kbar]$.
These are related by the formula $[L:K]=e_{L/K}\cdot
f_{L/K}$.  Say that $L/K$ is totally ramified if
$e_{L/K}=[L:K]$, and unramified if $e_{L/K}=1$.

     Let $L/K$ be a finite Galois extension, set
$G=\Gal(L/K)$, and let $G_0$ be the inertia subgroup of
$G$.  For real $x\ge0$ set
\[G_x=\{\sigma\in G_0:v_L(\sigma(\pi_L)-\pi_L)\ge x+1\}.\]
Then $G_x$ is the $x$th ramification subgroup of $G$
with the lower numbering.  We have $G_y\le G_x$ for
$y\ge x$, and $G_x=\{1_G\}$ for $x$ sufficiently large.
We say that $b\ge0$ is a lower ramification break of
$L/K$ if $G_b\ne G_{b+\epsilon}$ for all real
$\epsilon>0$.  The lower numbering of ramification
groups is compatible with subgroups in the following
sense: Let $M/K$ be a subextension of $L/K$ and set
$H=\Gal(L/M)$.  Then $H_x=H\cap G_x$ for $x\ge0$.

     In order to define a numbering on ramification
groups which is compatible with quotients we define the
Hasse-Herbrand function
$\phi_{L/K}:\R_{\ge0}\ra\R_{\ge0}$ by
\[\phi_{L/K}(x)=\int_0^x\frac{dt}{|G_0:G_t|}.\]
Then $\phi_{L/K}$ is a continuous bijection, so we may
define $\psi_{L/K}:\R_{\ge0}\ra\R_{\ge0}$ by
$\psi_{L/K}=\phi_{L/K}^{-1}$.  For $x\ge0$ define the
$x$th ramification subgroup of $G$ with the upper
numbering to be $G^x=G_{\psi_{L/K}(x)}$.  We say that
$u\ge0$ is an upper ramification break of $L/K$ if
$G^u\ne G^{u+\epsilon}$ for all real $\epsilon>0$.  Let
$M/K$ be a Galois subextension of $L/K$ and set
$H=\Gal(L/M)$.  It follows from Herbrand's theorem that
$(G/H)^x=G^xH/H$ for $x\ge0$ (see Proposition~14 in
\cite[IV]{cl}).

     The ramification breaks of Artin-Schreier
extensions are easy to compute:

\begin{lemma} \label{AS}
Let $a\in K$ and set $m=-v_K(a)$.  Assume that $m\ge1$
and $p\nmid m$.  Let $L$ be the splitting field over $K$
of the Artin-Schreier polynomial $X^p-X-a$.  Then $L/K$
is a totally ramified $(\Z/p\Z)$-extension and $m$ is the
unique lower ramification break of $L/K$.
\end{lemma}

\begin{proof}
Let $x$ be a root of $X^p-X-a$.  Then $v_K(x)=-p^{-1}m$,
so $x\not\in K$.  Hence $X^p-X-a$ is irreducible and
$L/K$ is a $(\Z/p\Z)$-extension.  Furthermore, there is
$\sigma\in\Gal(L/K)$ such that $(\sigma-1)x=1$.  Then
$v_L((\sigma-1)x)-v_L(x)=m$ with $p\nmid m$, so the
ramification break of $L/K$ is $m$.
\end{proof}

     Now suppose $L/K$ is a finite totally ramified
cyclic extension of degree $p^n$.  Then $L/K$ has $n$
distinct lower ramification breaks $b_1<b_2<\dots<b_n$,
and $n$ distinct upper ramification breaks
$u_1<u_2<\dots<u_n$.  The lower breaks are positive
integers since $L/K$ is a totally ramified Galois
extension of degree $p^n$.  Since $L/K$ is abelian the
upper breaks are also positive integers by the Hasse-Arf
theorem (see Theorem~1 in \cite[V\,\S7]{cl}).  We will
need the following well-known facts about the
ramification breaks of cyclic $p$-extensions:

\begin{lemma} \label{cyc}
Let $L/K$ be a totally ramified cyclic extension of
degree $p^n$.  Then
\begin{enumerate}[(a)]
\item $b_1=u_1$ and $b_{i+1}-b_i=p^i(u_{i+1}-u_i)$ for
$1\le i\le n-1$.
\item Let $M/K$ be the subextension of $L/K$ of degree
$p^i$.  Then the upper breaks of $M/K$ are
$u_1,\dots,u_i$, the lower breaks of $M/K$ are
$b_1,\dots,b_i$, and the lower breaks of $L/M$ are
$b_{i+1},\dots,b_n$.
\end{enumerate}
\end{lemma}

\begin{proof}
The statements in (a) follow from the definition of the
upper breaks.  The first statement in (b) is the
compatibility of the upper ramification numbering with
quotients, and the third statement is the compatibility
of the lower numbering with subgroups.  The second
statement follows from the first statement and part (a).
\end{proof}

     Following \cite{lt} we say that $\av\in W_n(K)$ is
reduced if for each $0\le i\le n-1$ we have either
$v_K(a_i)\ge0$ or $p\nmid v_K(a_i)$.  It is known (and
easy to verify) that for every $\bv\in W_n(K)$ there is
$\cv\in W_n(K)$ such that
$\bv\oplus F(\cv)\ominus\cv$ is reduced (see
Proposition~4.1 of \cite{lt}).  We say that $\av$ is
strongly reduced if for each $0\le i\le n-1$ we have
either $a_i=0$ or $p\nmid v_K(a_i)$.

     Kanesaka-Sekiguchi \cite{ks} and Thomas \cite{lt}
used class field theory to prove the following result in
the case where $K$ has finite residue field:

\begin{theorem} \label{main}
Let $K$ be a local field of characteristic $p$ with
perfect residue field and let $\av\in W_n(K)$ be a
reduced Witt vector.  For $0\le i\le n-1$ set
$m_i=-v_K(a_i)$, and assume that $m_0>0$.  Then $K_n/K$
is a totally ramified $(\Z/p^n\Z)$-extension, and for
$1\le i\le n$ the $i$th upper ramification break of
$K_n/K$ is
\begin{equation} \label{ui}
u_i=\max\{p^{i-1}m_0,p^{i-2}m_1,\dots,pm_{i-2},
m_{i-1}\}.
\end{equation}
\end{theorem}

\begin{proof}
Since $\wp(\M_K)=\M_K$ there is a finite unramified
Galois extension $K'/K$ and $\dv\in W_n(\OO_{K'})$ such
that $\av'=\av\oplus F(\dv)\ominus\dv$ is strongly
reduced.  The upper ramification breaks of $LK'/K'$ are
the same as those of $L/K$, and $v_{K'}(a_i')=v_K(a_i)$
for all $0\le i<n$ such that $m_i=-v_K(a_i)>0$.  Therefore
by replacing $\av$ with $\av'$ we may assume that $\av$
is strongly reduced.  Let $\xv\in W_n(K^{sep})$ satisfy
$F(\xv)=\xv\oplus\av$.  For $1\le i\le n$ set
$K_i=K(x_0,\dots,x_{i-1})$.  Then $K_n=K[\av]$ is
the extension of $K$ determined by $\chi_{\av}^K$.

     To prove the theorem we use induction on $n$.  The
case $n=1$ is given by Lemma~\ref{AS}.  Now let $n\ge2$
and assume that for
$1\le r<n$ the claim holds for strongly reduced Witt
vectors $\av\in W_r(K)$.  Then $K_{n-1}/K$ has upper
ramification breaks $u_1,\dots,u_{n-1}$ given by
(\ref{ui}).  By Lemma~\ref{cyc}(b) these are also upper
ramification breaks for $K_n/K$.  Therefore to prove the
theorem for $K_n/K$ it suffices to show that
\begin{equation} \label{suff}
u_n=\max\{p^{n-1}m_0,p^{n-2}m_1,\dots,pm_{n-2},
m_{n-1}\}
\end{equation}
is an upper ramification break of $K_n/K$.  We will do
this by first determining the largest ramification
breaks of $M=K[\mu_1(\av)]$ and $L=K[\lambda_1(\av)]$.
Recall that $K_n$ is contained in the compositum of
these two fields.

     We begin with $M=K[\mu_1(\av)]$.  Let $\xv\in
W_n(K^{sep})$ satisfy $F(\xv)=\xv\oplus\mu_1(\av)$.  It
follows from Proposition~\ref{subext} that
$\chi_{\mu_1(\av)}^K\big|_{G_{K_1}}=\chi_{\cv}^{K_1}$, with
\begin{align*}
\cv&=\lambda_1(\mu_1(\xv)\oplus\mu_1(\av)) \\
&=(0,S_1(x_0,0,a_0,0),S_2(x_0,0,0,a_0,0,0),\dots,
S_{n-1}(x_0,0,\dots,0,a_0,0\dots,0)).
\end{align*}
Let $1\le j\le n-1$ and set
\[f_j(X_0,Y_0)=S_j(X_0,0,\dots,0,Y_0,0\dots,0).\]
It follows from the recursive definition of $S_j$ that
$-X_0Y_0^{p^j-1}$ is a term in $f_j$.  Furthermore, the
total degree of $f_j$ is $p^j$, and the coefficient of
$Y_0^{p^j}$ in $f_j$ is 0.  Since $v_{K_1}(a_0)=-pm_0$ and
$v_{K_1}(x_0)=-m_0$ it follows that the unique term in
$f_j(x_0,a_0)$ with minimum valuation is
$-x_0a_0^{p^j-1}$.  Thus
\begin{align*}
v_{K_1}(S_j(x_0,0,\dots,0,a_0,0\dots,0))
&=v_{K_1}(f_j(x_0,a_0)) \\
&=v_{K_1}(-x_0a_0^{p^j-1}) \\
&=-(p^{j+1}-p+1)m_0.
\end{align*}
Since $p\nmid m_0$ this implies that the Witt vector
$\cv$ is strongly reduced.  It follows from the inductive
hypothesis that for $1\le i\le n-1$ the $i$th upper
ramification break of $K_1[\cv]/K_1$ is
\[\max\{p^{i-j}(p^{j+1}-p+1)m_0:1\le j\le i\}
=(p^{i+1}-p+1)m_0.\]
Hence by Lemma~\ref{cyc}(a), for $1\le i\le n-1$ the
$i$th lower ramification break of $K_1[\cv]/K_1$ is
\[b_{i+1}:=(p^{2i}-p^{2i-1}+\dots+p^2-p+1)m_0.\]
Using Lemma~\ref{cyc}(b) we find that for $2\le i\le n$,
$b_i$ is the $i$th lower ramification break of
$K_1[\cv]/K$.  In addition, by Lemma~\ref{AS} $K_1/K$
has lower ramification break $b_1=m_0$.  Hence by
Lemma~\ref{cyc}(b), $K_1[\cv]/K$ also has lower ramification
break $m_0$.  Therefore by Lemma~\ref{cyc}(a) the upper
ramification breaks of $K_1[\cv]/K$ are
$m_0,pm_0,\dots,p^{n-1}m_0$, in agreement with
(\ref{ui}).

     Now we consider $L=K[\lambda_1(\av)]$.
If $\lambda_1(\av)=\zerov$ then $L/K$ is the trivial
extension and there is nothing to do.  Suppose
$\lambda_1(\av)=(0,\dots, 0,a_r,\dots,a_{n-1})\in
W_n(K)$ for some $1\le r\le n-1$ with $a_r\ne 0$.  Then
$(a_r,\dots,a_{n-1})\in W_{n-r}(K)$ is strongly reduced.
Using Lemma~\ref{shift} and induction we find that the
largest upper ramification break of $L/K$ is
\begin{align*}
w&=\max\{p^{n-1-r}m_r,p^{n-2-r}m_{r+1},\dots,pm_{n-2},m_{n-1}\}
\\
&=\max\{p^{n-2}m_1,p^{n-3}m_2,\dots,pm_{n-2},m_{n-1}\}.
\end{align*}

     Set $H=\Gal(LM/K)$.  Then $\chi_{\mu_1(\av)}^K$ and
$\chi_{\lambda_1(\av)}^K$ factor through $H$, so they
may be viewed as homomorphisms from $H$ to $\Z/p^n\Z$.
Since $p^{n-1}\nmid w$ we have either $w>p^{n-1}m_0$ or
$w<p^{n-1}m_0$.  Suppose $w>p^{n-1}m_0$.  By the
compatibility of the upper ramification numbering with
quotients we get
$\chi_{\lambda_1(\av)}^K(H^w)\ne \{0+p^n\Z\}$,
$\chi_{\lambda_1(\av)}^K(H^{w+\epsilon})=\{0+p^n\Z\}$,
and $\chi_{\mu_1(\av)}^K(H^w)=\{0+p^n\Z\}$.  Since
$\av=\mu_1(\av)\oplus\lambda_1(\av)$ we have
$\chi_{\av}^K=\chi_{\mu_1(\av)}^K
+\chi_{\lambda_1(\av)}^K$.  Hence
$\chi_{\av}^K(H^w)\ne\{0+p^n\Z\}$ and
$\chi_{\av}^K(H^{w+\epsilon})=\{0+p^n\Z\}$, which shows
that $w$ is an upper ramification break of $K_n/K$.
Similarly, if $w<p^{n-1}m_0$ then $p^{n-1}m_0$ is an
upper ramification break of $K_n/K$.  In either case it
follows that $u_n=\max\{p^{n-1}m_0,w\}$ is an upper
ramification break of $K_n/K$.
\end{proof}

\begin{cor}
Let $\av\in W_n(K)$ be a reduced Witt vector and set
$m_i=-v_K(a_i)$ for $0\le i\le n-1$.  Assume $m_0=0$ and
$a_0\not\in\wp(K)=\{c^p-c:c\in K\}$.  Let $1\le r\le n$
be maximum such that $m_i\le0$ for $0\le i<r$.  Let
$K_n=K[\av]$.
Then $K_n/K$ is a $(\Z/p^n\Z)$-extension with residue
class degree $p^r$ and ramification index $p^{n-r}$.
Furthermore, the positive upper ramification breaks of
$K_n/K$ are $u_1,\dots,u_{n-r}$, where
\begin{equation} \label{uicor}
u_i=\max\{p^{i-1}m_r,p^{i-2}m_{r+1},\dots,pm_{r+i-2},
m_{r+i-1}\}.
\end{equation}
\end{cor}

\begin{proof}
It follows from the assumptions that $X^p-X-a_0$ is
irreducible over $K$.  Hence
$\chi_{\av}^K(G_K)=\Z/p^n\Z$ and $K_n/K$ is a
$(\Z/p^n\Z)$-extension.  Set $L=K[\lambda_r(\av)]$,
$M=K[\mu_r(\av)]$, and $H=\Gal(LM/K)$.  Then
$\chi_{\lambda_r(\av)}^K$ and $\chi_{\mu_r(\av)}^K$
factor through $H$, so they may be viewed as
homomorphisms from $H$ to $\Z/p^n\Z$.  We have
$\av=\mu_r(\av)\oplus\lambda_r(\av)$, and hence
$\chi_{\av}^K=\chi_{\mu_r(\av)}^K+\chi_{\lambda_r(\av)}^K$.
Using Hensel's lemma we see that $M/K$ is unramified.
Therefore for $\sigma\in H_0$ we have
$\chi_{\mu_r(\av)}^K(\sigma)=0+p^n\Z$, and hence
$\chi_{\av}^K(\sigma)=\chi_{\lambda_r(\av)}^K(\sigma)$.
Using Lemma~\ref{shift} and Theorem~\ref{main} we find
that $L/K$ has ramification index $p^{n-r}$ and upper
ramification breaks $u_1,\dots u_{n-r}$ given by
(\ref{uicor}).  Since $\chi_{\av}^K\big|_{H_0}
=\chi_{\lambda_r(\av)}^K\big|_{H_0}$ it follows that
$K_n/K$ also has ramification index $p^{n-r}$ and upper
breaks $u_1,\dots u_{n-r}$.  Since $[K_n:K]=p^n$ we
deduce that $K_n/K$ has residue class degree
$p^n/p^{n-r}=p^r$.
\end{proof}

\begin{remark}
If the residue class degree of $K_n/K$ is greater than 1
then $-1$ is an upper ramification break of $K_n/K$.
\end{remark}


\begin{thebibliography}{9}

\bibitem{pdg} M. Demazure, {\em Lectures on $p$-Divisible
Groups}, Lecture Notes in Mathematics, Vol.\ 302,
Springer-Verlag, Berlin-New York,  1972.

\bibitem {ks} K. Kanesaka, and K. Sekiguchi,
Representation of Witt vectors by formal power series
and its applications, Tokyo J. Math.\ {\bf2} (1979),
349--370.

\bibitem{cl} J.-P. Serre, Local fields, Translated from
the French by Marvin Jay Greenberg, Graduate Texts in
Mathematics 67, Springer-Verlag, New York-Berlin, 1979.

\bibitem{ltthesis} L. Thomas, Arithm\'{e}tique des
extensions d’Artin-Schreier-Witt, thesis, Universit\'e
de Toulouse II le Mirail, 2005.

\bibitem{lt} L. Thomas, Ramification groups in
Artin-Schreier-Witt extensions, J. Theor.\ Nombres
Bordeaux {\bf17} (2005), 689--720.

\bibitem{W} E. Witt, Zyklische K\"orper und Algebren der
Charakteristik $p$ vom Grad $p^n$, J. Reine Angew.\
Math.\ {\bf174} (1936), 126--140.

\end{thebibliography}
\end{document}